\pdfoutput=1
\documentclass{amsart}
\usepackage{graphicx}

\usepackage{tabularx}

\newtheorem{theorem}{Theorem}
\newtheorem{proposition}{Proposition}

\theoremstyle{definition}
\newtheorem{example}{Example}
\newtheorem{definition}{Definition}

\begin{document}
\title{Strong and weak (1, 2, 3) homotopies on knot projections}
\thanks{MSC2010, MSC2000: 57M25; 57Q35.\\
The work of N.~Ito was partly supported by a Waseda University Grant for Special Research Projects (Project number: 2014K-6292) and the JSPS Japanese-German Graduate Externship.}
\author{Noboru Ito}
\address{Waseda Institute for Advanced Study, 1-6-1 Nishi Waseda Shinjuku-ku Tokyo 169-8050 Japan}
\email{noboru@moegi.waseda.jp}
\author{Yusuke Takimura}
\address{Gakushuin Boys' Junior High School, 1-5-1 Mejiro Toshima-ku Tokyo 171-0031, Japan}
\email{Yusuke.Takimura@gakushuin.ac.jp}
\date{June 6, 2015}
\maketitle

\begin{abstract}
A knot projection is an image of a generic immersion from a circle into a two-dimensional sphere.  We can find homotopies between any two knot projections by local replacements of knot projections of three types, called Reidemeister moves.    
This paper defines an equivalence relation for knot projections called weak (1, 2, 3) homotopy, which consists of Reidemeister moves of type 1, weak type 2, and weak type 3.  This paper defines the first non-trivial invariant under weak (1, 2, 3) homotopy.   
We use this invariant to show that there exist an infinite number of weak (1, 2, 3) homotopy equivalence classes of knot projections.  
By contrast, all equivalence classes of knot projections consisting of the other variants of a triple type, i.e., Reidemeister moves of (1, strong type 2, strong type 3), (1, weak type 2, strong type 3), and (1, strong type 2, weak type 3), are contractible.    
%By contrast, all knot projections under strong (1, 2, 3) homotopy are contractible.  Similarly, equivalence classes consisting of the other choice of a triple type, namely, Reidemeister moves of (type 1, strong type 2, weak type 3) or (type 1, weak type 2, strong type 3) are contractible.   
\end{abstract}

\section{Introduction.}
A {\it{knot projection}} is defined as an image of a generic immersion of a circle into a two-dimensional sphere.   A {\it{knot diagram}} is a knot projection specifying information of over/under-crossing branches.  
A {\it{trivial}} knot projection is defined as a knot projection with no double points, namely a simple closed curve on a $2$-sphere.  It is well known that an arbitrary knot projection can be related to a trivial knot projection by through RI (Reidemeister move of type 1), RI\!I (Reidemeister move of type 2), and RI\!I\!I (Reidemeister move of type 3), as shown in Fig.~\ref{f1}.  Here, each crossing in Fig.~\ref{f1} is a {\it{flat crossing}}.  A flat crossing is a double point of a knot diagram without over/under information.  
\begin{figure}[h!]
\includegraphics[width=8cm]{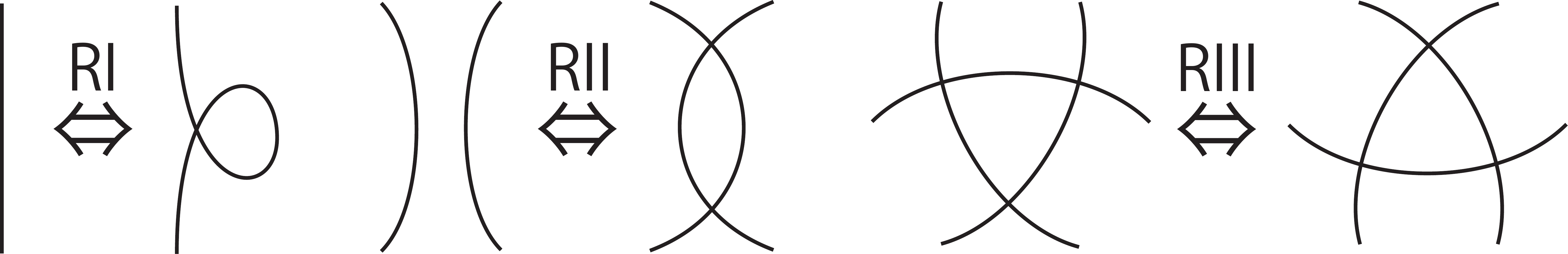}
\caption{RI, RI\!I, and RI\!I\!I.}\label{f1}
\end{figure}

RI\!I (resp.~RI\!I\!I) can be decomposed into exactly two types, namely, strong RI\!I (resp.~strong RI\!I\!I) and weak RI\!I (resp.~weak RI\!I\!I), as shown in Fig.~\ref{f1a}.
\begin{figure}[h!]
\includegraphics[width=10cm]{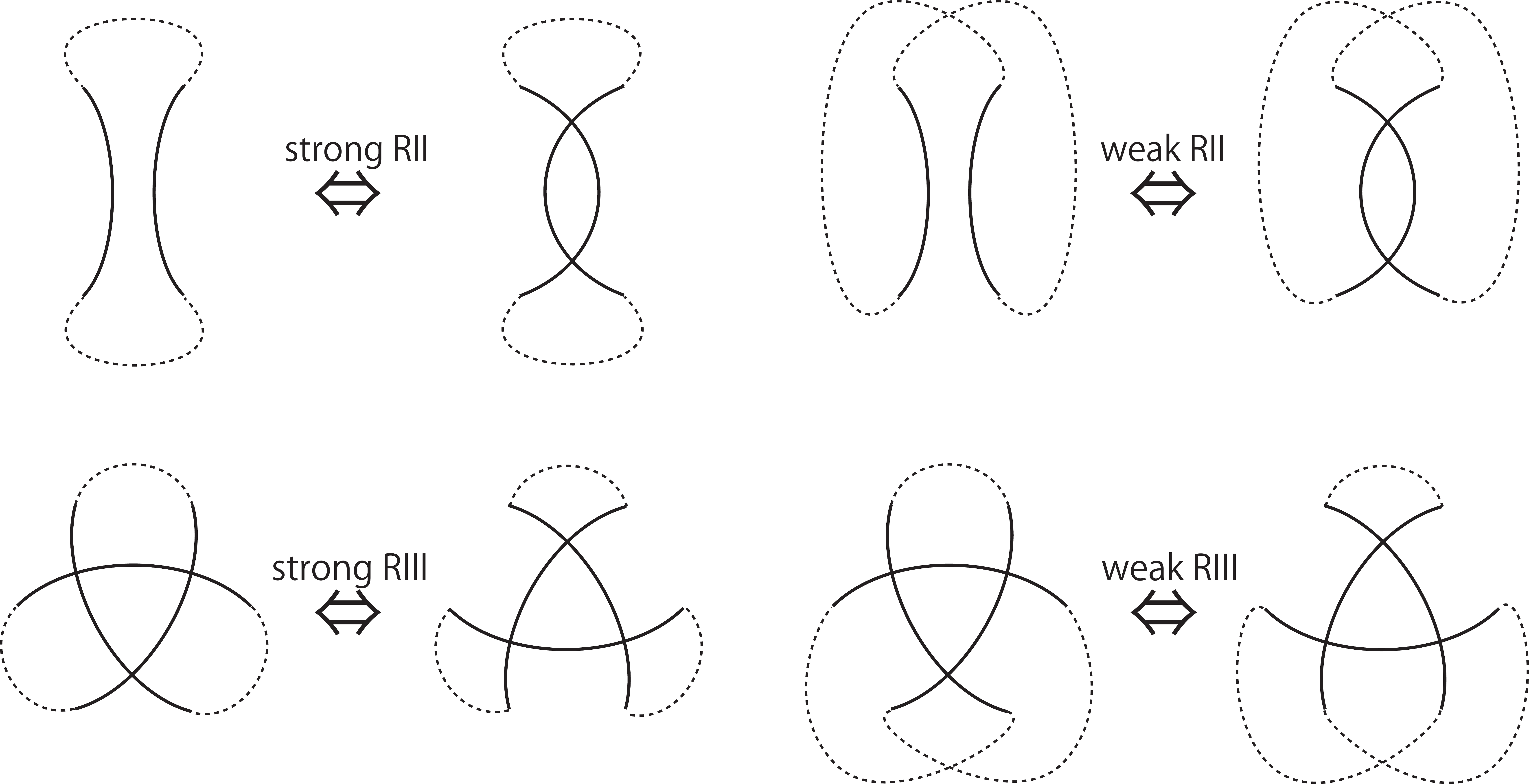}
\caption{Strong RI\!I and weak RI\!I (upper), strong RI\!I\!I, and weak RI\!I\!I (lower).}\label{f1a}
\end{figure}
Here, we introduce equivalence relations, called {\it{strong}} (1, 2, 3) {\it{homotopy}} and {\it{weak}} (1, 2, 3) {\it{homotopy}}, which are defined as follows.       
Two knot projections are said to be {\it{strongly}} (resp.~{\it{weakly}}) (1, 2, 3) {\it{homotopic}} if the two knot projections are related by a finite sequence consisting of RI, strong RI\!I (resp.~weak RI\!I), and strong  RI\!I\!I (resp.~weak RI\!I\!I).   
\begin{theorem}\label{thm1}

\noindent(1) Any two knot projections are strongly (1, 2, 3) homotopic.  
 
\noindent(2) There exists infinitely many knot projections $\{ P_i \}$ that are not weakly (1, 2, 3) homotopic to a trivial knot projection, where $P_i$ is not weakly (1, 2, 3) homotopic to $P_j$ for any two nonnegative integers $i$ and $j$, such that $i \neq j$.  
\end{theorem}
To prove Theorem~\ref{thm1} (2), we find a non-negative integer-valued new invariant $W$ under RI, weak RI\!I, and weak RI\!I\!I.  

Here, we explain the historical position of this study.  Equivalence classes of knot projections, on a plane and sphere, have been studied by many researchers.  For example, Arnold introduced invariants, $J^+$, $J^-$, and $St$, for plane curves \cite{arnold}.  However, almost all these studies treat plane curves (i.e., knot projections on a plane) under regular homotopy, i.e., without RI.  To the best of our knowledge, only $J^+ + 2St$, given by the Arnold's invariants, is invariant under RI for knot projections on a sphere.  Studies of equivalence classes of knot projections under homotopy containing RI are: \cite{khovanov, IT} for RI and RI\!I, \cite{IT_circle, IT_triple} for a pair of RI and strong or weak RI\!I, \cite{HY} for RI and RI\!I\!I, \cite{ITT} for RI and strong RI\!I\!I, and \cite{IT} for RI and weak RI\!I\!I.  

However, this is the first paper that considers equivalence classes consisting of a triple of Reidemeister moves: (RI, strong RI\!I, strong RI\!I\!I) and (RI, weak RI\!I, weak RI\!I\!I).  

The reminder of this paper is organized as follows.  Section~\ref{sec2} defines a new invariant $W$ under RI, weak RI\!I, and weak RI\!I\!I.  
%and introduces properties.  
%Section~\ref{sec3} obtains a proof of Theorem~\ref{thm1}. 
Section~\ref{sec4} introduces properties of $W$.  Section~\ref{sec3} obtains a proof of Theorem~\ref{thm1}.   Section~\ref{sec5} demonstrates that every knot projection under each equivalence relation consisting of the other variants of a triple tuple of Reidemeister moves, namely (RI, strong RI\!I, weak RI\!I\!I) or (RI, weak RI\!I, strong RI\!I\!I), is equivalent to a trivial knot projection.  

\section{Definition of a new invariant $W$.}\label{sec2}
Let $P$ be an arbitrary knot projection and the number of double points be denoted by $c(P)$.  {\it{Seifert circle number}} $s(P)$ is defined as follows.  
\begin{figure}[h!]
\includegraphics[width=4cm]{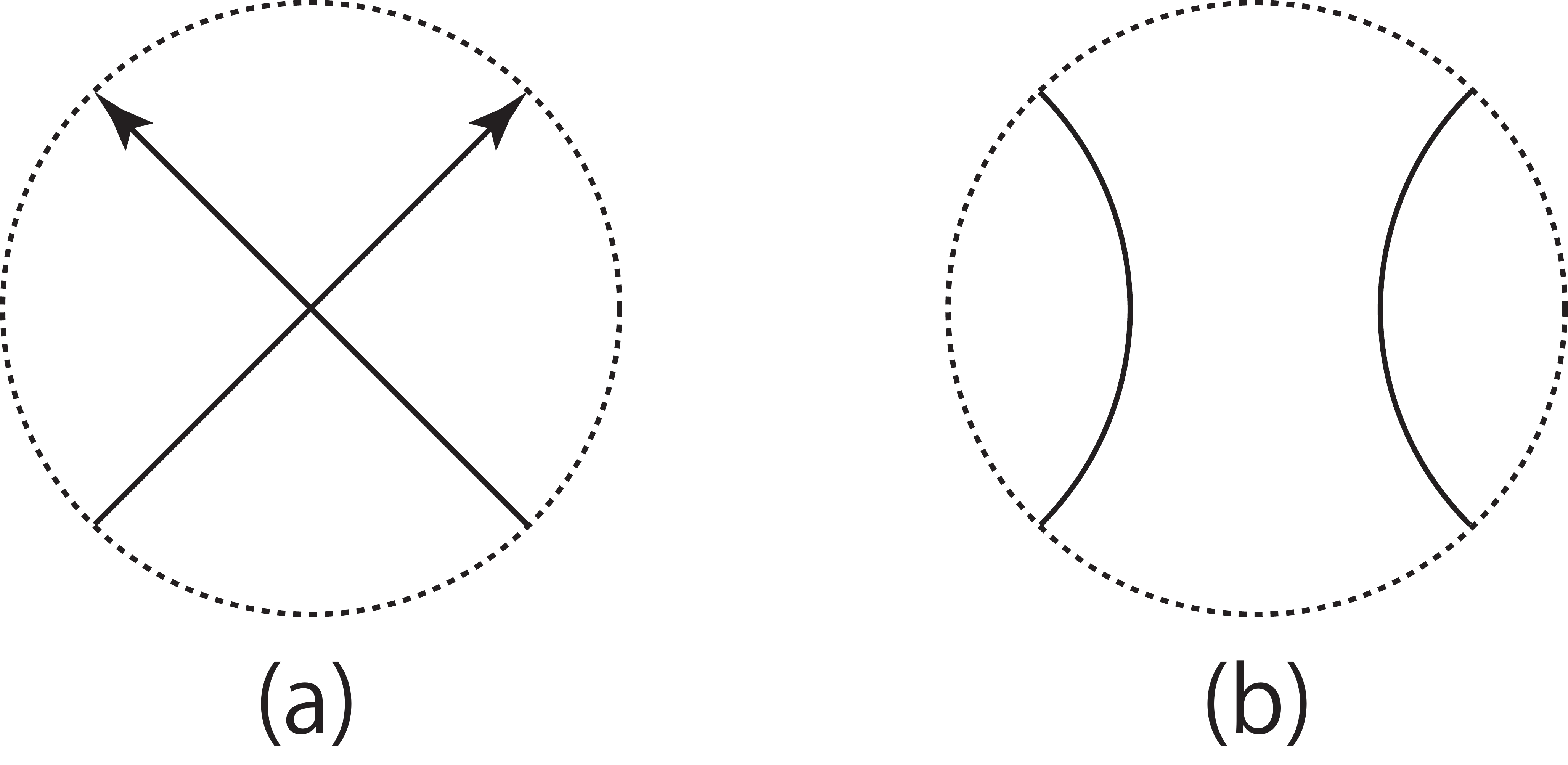}
\caption{(a) Disk as a neighborhood of a double point; (b) Disk corresponding to (a) after applying the replacement.}\label{f2}
\end{figure}
When $P$ has an arbitrary orientation $\sigma$, a sufficient small neighborhood of each double point is isotopic to Fig.~\ref{f2} (a).  We replace the disk in Fig.~\ref{f2} (a) with that in Fig.~\ref{f2} (b).  After replacing all the double points, we obtain an arrangement $S(P)$ of circles on the sphere.  By definition, $S(P)$ does not depend on the orientation $\sigma$ of $P$.  Each circle in $S(P)$ is called a {\it{Seifert circle}}.  Seifert circle number $s(P)$ is defined as the number of circles in $S(P)$.  It is well-known that for any knot diagram $D_P$ obtained by arbitrarily obtaining information of any over/under-crossing branches, the {\it{canonical genus}} of $D_P$ is equivalent to $c(P)-s(P)+1$; $c(P)-s(P)+1$ is denoted by $g(P)$.  Note that all knot diagrams obtained by the knot projection $P$ have the same canonical genus.   

A {\it{chord diagram}} of the knot projection $P$ is defined as the immersing circle with even points on the circle, which correspond to the preimages of double points of $P$ where each pair of preimages of a double point is connected by a chord.  The {\it{trivializing number}} $tr(P)$ of the knot projection $P$ is the minimum number of erased chords in the chord diagram of $P$ until there is no cross chord, i.e., there are no intersecting chords as $\otimes$, embedded in the whole chord diagram  
% that is sub-chords as $\otimes$  %consisting of two intersecting chords embedded in a whole chord diagram 
\cite[Page 440, Theorem 13]{hanaki}.  
%Here, a {\it{cross chord}} is a part as $\otimes$ of a chord diagram of $P$.  
\begin{theorem}\label{thm2}
Let $P$ be an arbitrary knot projection.  Then,
\[W(P) = tr(P) - 2g(P) \] is invariant under RI, weak RI\!I, and weak RI\!I\!I.    
\end{theorem}
\begin{proof}
We check the differences $tr(P)$, $s(P)$, and $c(P)$ because $W(P)$ $=$ $tr(P)-2g(P)$ $=$ $tr(P)$ $+$ $s(P)$ $-$ $c(P)$ $-1$.  
By definitions, it is clear that $c(P)$ and $s(P)$ are invariant under weak RI\!I\!I.  We also recall that there is a fact that $tr(P)$ is invariant under RI and weak RI\!I\!I (cf.~\cite{ITT}).   
Thus, we check the differences before and after the application RI and weak RI\!I (Table~\ref{t1}).  A single RI (resp.~weak RI\!I) with increasing $c(P)$ is denoted by $1a$ (resp.~$w2a$).  
\begin{table}[h!]
\caption{The differences before and after the application of $1a$ and $w2a$.}\label{t1}
\begin{tabular}{|c|c|c|c|} \hline
& $tr(P)$ & $s(P)$ & $c(P)$ \\ \hline
$1a$ & $0$ & $+1$ & $+1$ \\ \hline
$w2a$ & $+2$ & $0$ & $+2$ \\ \hline
%weak RI\!I\!I & $0$ & $0$ & $0$  \\ \hline
\end{tabular}
\end{table}
\end{proof}
%\section{Proof of Theorem~\ref{thm1}.}\label{sec3}
%\begin{proof}
%{\bf{Proof of (1).}} It is sufficient to show that a single weak RI\!I and a single weak RI\!I\!I are generated by a finite sequence of consisting of RI, strong RI\!I and strong RI\!I\!I by using Fig.~\ref{f3}.  
%\begin{figure}[h!]
%\includegraphics[width=10cm]{f3.pdf}
%\caption{A single weak RI\!I\!I consists of two strong RI\!Is and a single strong RI\!I\!I (upper line).  A single weak RI\!I consists of two RIs, a single strong RI\!I, and a single weak RI\!I\!I (lower line).}\label{f3}
%\end{figure}

%{\bf{Proof of (2).}} Let $O$ be a trivial knot projection, namely a simple closed curve.  Here, $W(O)=0$ and $W(7_4)$ $=$ $2$.  Since $W(P)$ holds the property of Theorem~\ref{thm3} (\ref{3}), we have the claim.  
%\end{proof}
\begin{example}\label{eg1}
\begin{figure}[h!]
\includegraphics[width=5cm]{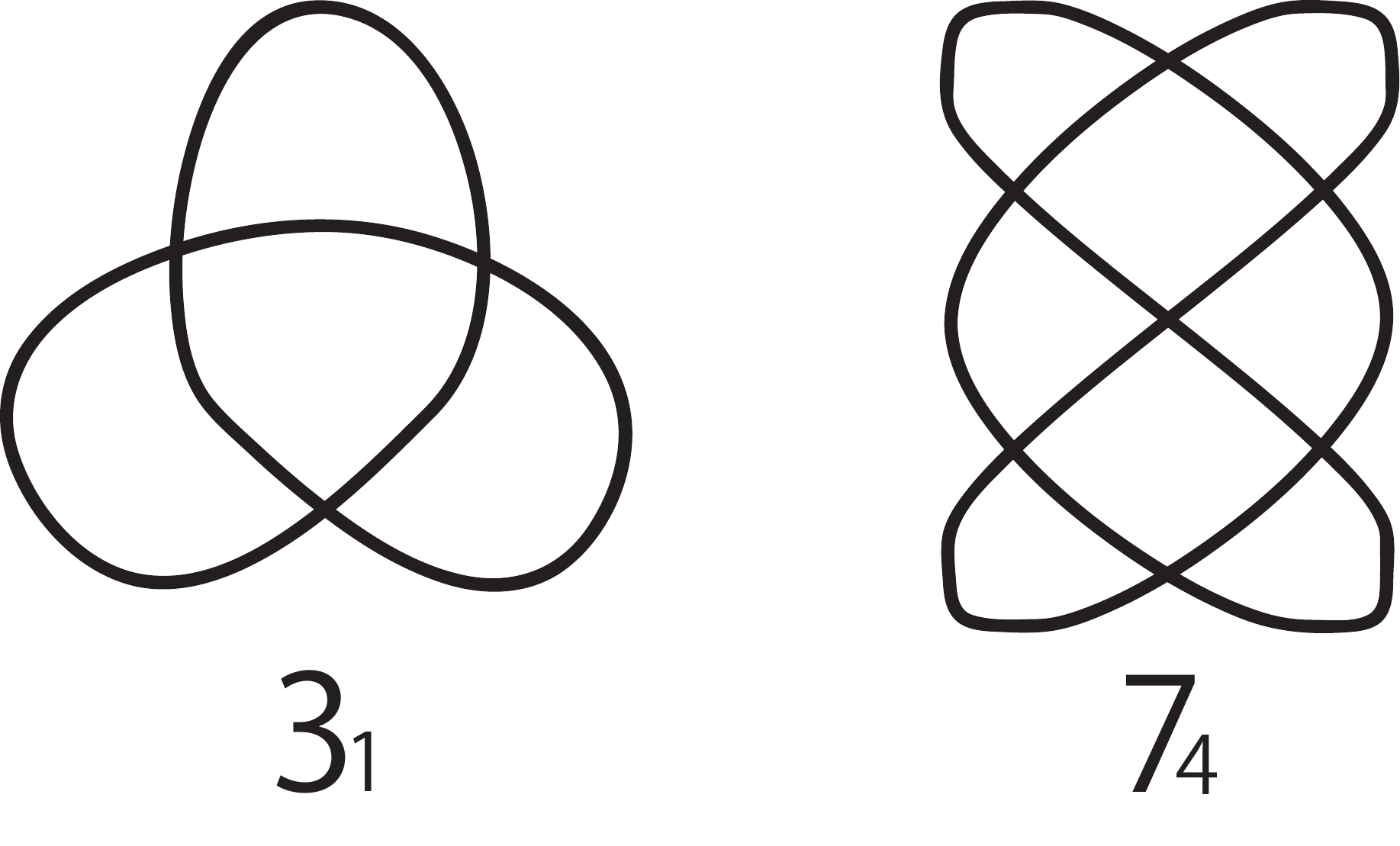}
\caption{$3_1$ and $7_4$.}\label{f3b}
\end{figure}
Let $3_1$ and $7_4$ be knot projections, as shown in Fig.~\ref{f3b}.  $W(3_1)$ $=$ $tr(3_1)$ $+$ $s(3_1)$ $-$ $c(3_1)$ $-1$ $=$ $2$ $+$ $2$ $-$ $3$ $-1$ $=$ $0$.  $W(7_4)$ $=$ $tr(7_4)$ $+$ $s(7_4)$ $-$ $c(7_4)$ $-1$ $=$ $4$ $+$ $6$ $-$ $7$ $-1$ $=$ $2$.  
\end{example}
\section{Properties of $W$.}\label{sec4}
In this section, we introduce the properties of $W$.  The {\it{connected sum}} of $P_1$ and $P_2$ is defined in Fig.~\ref{f4}.
\begin{figure}[h!]
\includegraphics[width=8cm]{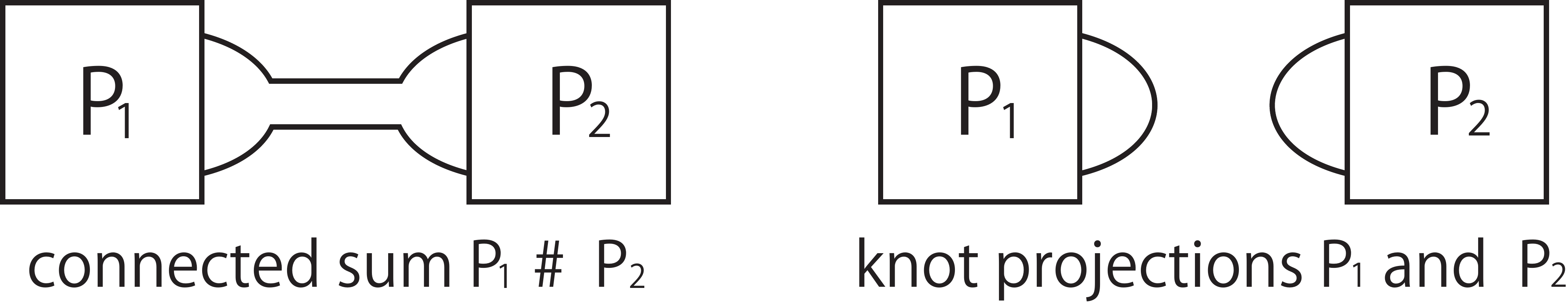}
\caption{Connected sum $P_1 \sharp P_2$ for knot projections $P_1$ and $P_2$.}\label{f4}
\end{figure}
If a knot projection $P$ is non-trivial and cannot be a connected sum of non-trivial knot projections, $P$ is called a {\it{prime}} knot projection.  
%%prime_def_revised_as_above%%%%
%If a knot projection $P$ is decomposed into two sub-curves $a_1$ and $a_2$, that intersect at exactly two double points by any simple closed circle and either $a_1$ or $a_2$ is a simple arc, then $P$ is said to be {\it{prime}}.
\begin{theorem}\label{thm3}
Let $P$ and $P'$ be arbitrary two knot projections.  Then, $W(P)$ has the following properties.  
\begin{enumerate}
\item $W(P)$ is an even integer.  \label{1}
\item $W(P \sharp P')$ $=$ $W(P) + W(P')$.\label{2}
%\item For an arbitrary negative odd integer $k$, there exist a knot projection $P$ such that $\$
\item For an arbitrary even nonnegative integer $m$, there exists a knot projection $P$ such that $W(P)=m$.  Further, a prime knot projection $P$ can be chosen.  \label{3}
\item $0 \le W(P) \le c(P)-1$ for an arbitrary non-trivial knot projection $P$.  
\item If $W(P)=c(P)-1$, $P$ is the knot projection that appears as $\infty$.  
\end{enumerate}
\end{theorem}
\begin{proof}
\begin{enumerate}
\item $tr(P)$, $s(P)$, and $c(P)$ change by $\pm 2$ or $0$ under a single strong RI\!I or a single strong RI\!I\!I.  Thus, we have the claim.  
\item It is clear that $tr(P \sharp P')$ $=$ $tr(P)$ $+$ $tr(P')$, $s(P \sharp P')$ $=$ $s(P)$ $+$ $s(P')$ $-1$, and $c(P \sharp P')$ $=$ $c(P)$ $+$ $c(P')$.   
\item Note that $W(3_1)=0$ and $W(7_4)=2$ (see Example~\ref{eg1}).  Thus, by using Theorem~\ref{thm3}~(\ref{2}) and considering any connected sum, we have the former claim.  For the latter claim, consider a rational knot projection $p(\alpha_1, \alpha_2)$ with a pair of even integers $(\alpha_1, \alpha_2)$ such that $\alpha_1 \ge \alpha_2 \ge 4$, as shown in Fig.~\ref{f5}, where $\alpha_1$ and $\alpha_2$ represent the number of double points.  We have $tr(p(\alpha_1, \alpha_2))$ $=$ $\alpha_2$, $s(p(\alpha_1, \alpha_2))$ $=$ $\alpha_1$ $+$ $\alpha_2$ $-1$, and $c(p(\alpha_1, \alpha_2))$ $=$ $\alpha_1+\alpha_2$.  Thus, $W(p(\alpha_1, \alpha_2))=\alpha_2 - 2$.   

Alternatively, consider a pretzel knot projection $q(\beta_1, \beta_2, \beta_3)$ with a tuple of odd integers $(\beta_1, \beta_2, \beta_3)$ such that $\beta_1 \ge \beta_2 \ge \beta_3 \ge 3$, as shown in Fig.~\ref{f6}, where $\beta_1$, $\beta_2$, and $\beta_3$ represent the number of double points.  We have $tr(q(\beta_1, \beta_2, \beta_3))$ $=$ $\beta_2$ $+$ $\beta_3$, $s(q(\beta_1, \beta_2, \beta_3))$ $=$ $\beta_1$ $+$ $\beta_2$ $+$ $\beta_3$ $-1$, and $c(q(\beta_1, \beta_2, \beta_3))$ $=$ $\beta_1$ $+$ $\beta_2$ $+$ $\beta_3$.  Thus, $W(q(\beta_1, \beta_2, \beta_3))=\beta_2 + \beta_3 -2$.  
%and q(b_1, b_2, b_3) with a tuple of odd integers $(b_1, b_2, b_3)$ defined by Fig.~\ref{f5}.
\begin{figure}[h!]
\includegraphics[width=5cm]{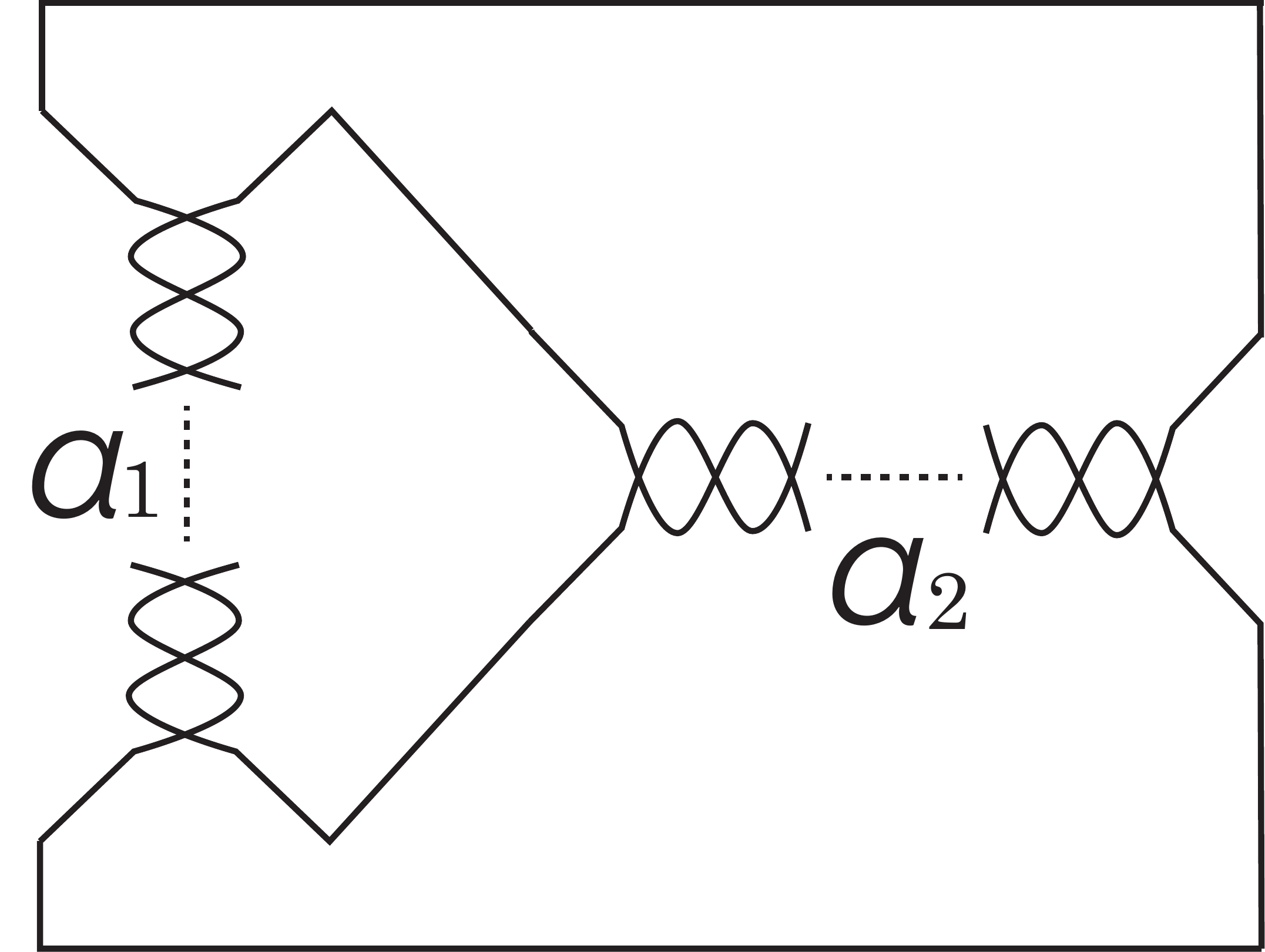}
\caption{$p(\alpha_1, \alpha_2)$.}\label{f5}
\end{figure}
\begin{figure}[h!]
\includegraphics[width=5cm]{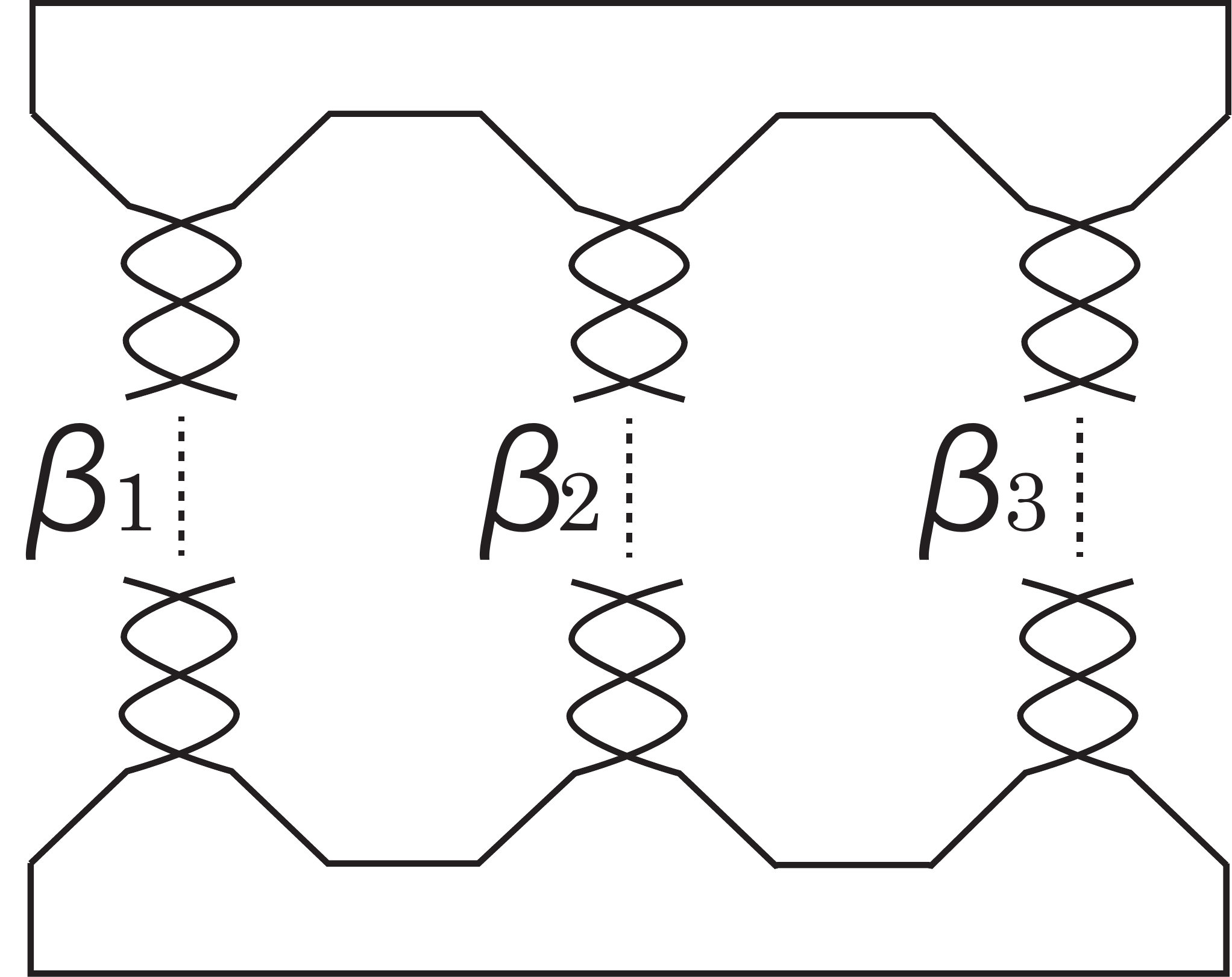}
\caption{$q(\beta_1, \beta_2, \beta_3)$.}\label{f6}
\end{figure}
\item $0 \le tr(P)-2g(P)$ ($=$ $W(P)$) is obtained by a known inequality (see \cite[Proof of Theorem 7.11]{henrich+}).  We can see that $s(P) \le c(P)+1$ and $tr(P) \le c(P)-1$ if $1 \le c(P)$.  Thus, \[W(P) = tr(P)+s(P)-c(P)-1 \le c(P)-1.\]
\item If $W(P)=c(P)-1$, $tr(P)=c(P)-1$, and $s(P)=c(P)+1$.  From \cite[Theorem 1.11]{hanakiOsaka}, if $tr(P)=c(P)-1$, $P$ is a knot projection as shown in Fig.~\ref{f5}.  In this case, $c(P)+1=s(P)=2$.  Then, $c(P)=1$.  
\end{enumerate}
\begin{figure}[h!]
\includegraphics[width=4cm]{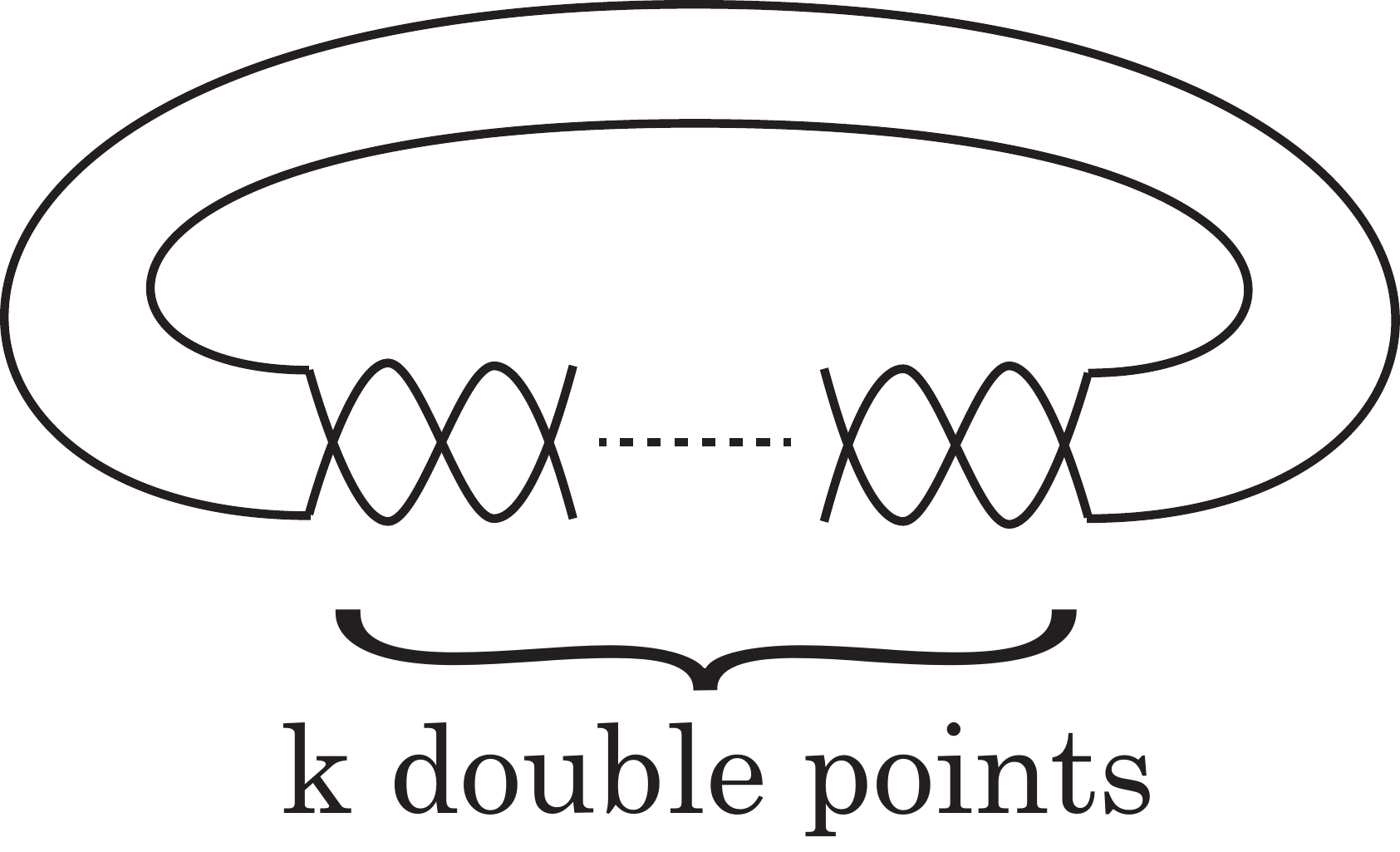}
\caption{A series of knot projections where $k$ is a positive odd integer.}
\end{figure}
\end{proof}
%\begin{remark}
%\end{remark}
\section{Proof of Theorem~\ref{thm1}.}\label{sec3}
\begin{proof}
{\bf{Proof of (1).}}  It can be seen that a single weak RI\!I and a single weak RI\!I\!I are generated by a finite sequence consisting of RI, strong RI\!I, and strong RI\!I\!I from Fig.~\ref{f3}.  

\begin{figure}[h!]
\includegraphics[width=8cm]{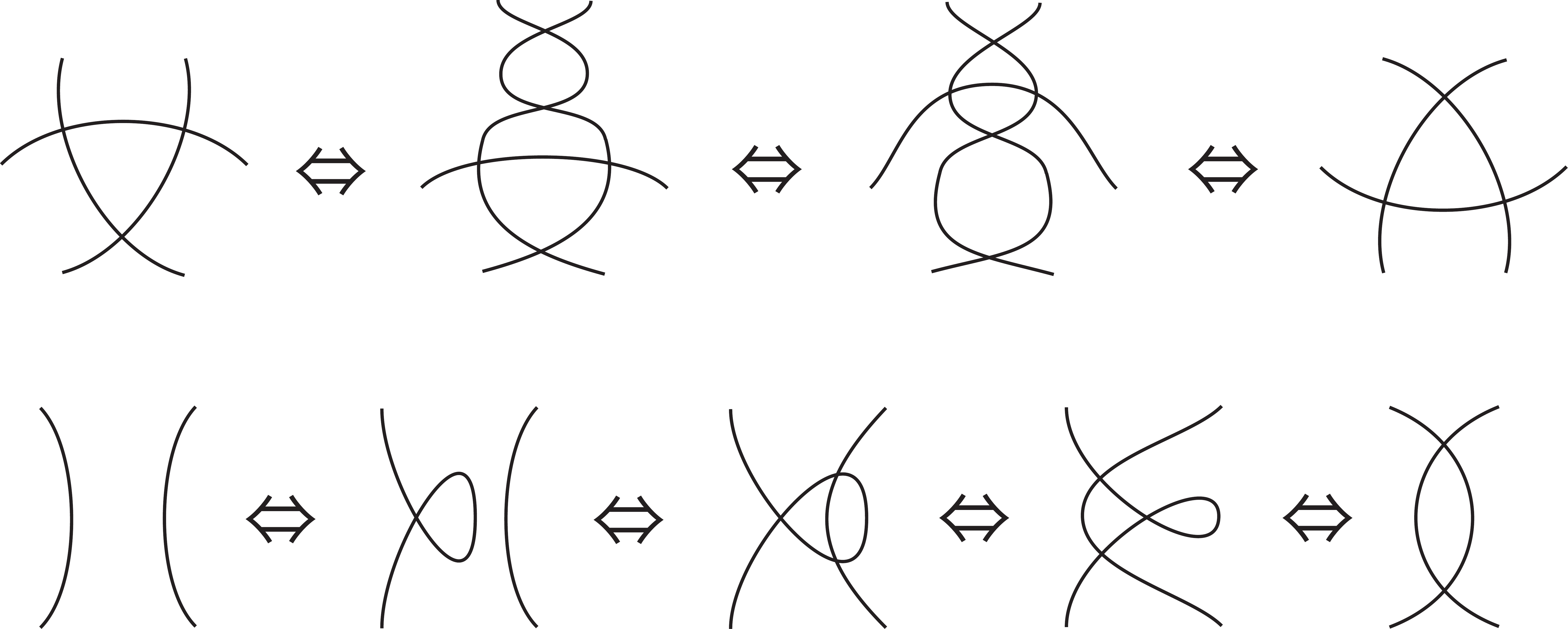}
\caption{A single RI\!I\!I consisting of two RI\!Is and a single RI\!I\!I (upper line).  A single RI\!I consisting of two RIs, a single RI\!I, and a single RI\!I\!I (lower line).}\label{f3}
\end{figure}

\noindent{\bf{Proof of (2).}}  Let $O$ be a trivial knot projection, i.e., a simple closed curve.  Here, $W(O)=0$ and $W(7_4)=2$.  Since $W(P)$ holds the property of Theorem~\ref{thm3}~(3), we have the claim.  
\end{proof}

\section{Other equivalences.}\label{sec5}
%In the last section, we remark that each equivalence classes of the other choices of triple types of local moves.  
\begin{definition}
(1) Two knot projections $P$ and $P'$ are ($1$, $s2$, $w3$) homotopic if $P$ and $P'$ are related by a finite sequence consisting of RI, strong RI\!I, and weak RI\!I\!I.  

\noindent(2) Two knot projections $P$ and $P'$ are ($1$, $w2$, $s3$) homotopic if $P$ and $P'$ are related by a finite sequence consisting of RI, weak RI\!I, and strong RI\!I\!I.  
\end{definition}
\begin{proposition}
%Let $O$ be a trivial knot projection, namely a simple closed curve.  

\noindent $(1)$ Any two knot projections are $(1, s2, w3)$ homotopic.  

\noindent $(2)$ Any two knot projections are $(1, w2, s3)$ homotopic.  
\end{proposition}
\begin{proof}
Recall Theorem~\ref{thm1} (1): an arbitrary knot projection can be related to a trivial knot projection $O$ by a finite sequence consisting of RI, strong RI\!I, and strong RI\!I\!I.  Thus, we verify that a single strong RI\!I\!I (resp.~strong RI\!I) is created in the case (1) (resp.~(2)) as follows.  

\noindent (1) From Fig.~\ref{f3} (upper line), we can see that a single strong RI\!I\!I consists of two strong RI\!I and a single weak RI\!I\!I.  

\noindent (2) From Fig.~\ref{f3} (lower line), we can see that a single strong RI\!I consists of two RIs, a single weak RI\!I and a single strong RI\!I\!I.  
\end{proof}

\section*{Acknowledgements.}
The authors would like to thank Professor Kouki Taniyama for his variable comments.

\end{document}